\documentclass[a4paper,11pt]{article}

\usepackage[latin1]{inputenc}
\usepackage[T1]{fontenc}
\usepackage[english]{babel}

\usepackage{mathrsfs}
\usepackage{amscd}
\usepackage{amsfonts}
\usepackage{amsmath}
\usepackage{amssymb}
\usepackage{amstext}
\usepackage{amsthm}
\usepackage{amsbsy}

\usepackage{xspace}
\usepackage[all]{xy}
\usepackage{graphicx}
\usepackage{url}
\usepackage{latexsym}

\usepackage{graphicx} 

\usepackage{booktabs} 
\usepackage{array} 
\usepackage{paralist} 
\usepackage{verbatim} 
\usepackage{subfig} 

\usepackage{fancyhdr} 
\usepackage{sectsty}
\allsectionsfont{\sffamily\mdseries\upshape} 

\usepackage[nottoc,notlof,notlot]{tocbibind} 
\usepackage[titles,subfigure]{tocloft} 

\usepackage[textwidth=100pt,textsize=footnotesize,bordercolor=white,color=blue!30]{todonotes}
\usepackage{hyperref} 

\makeatletter
\newcommand*{\rom}[1]{\expandafter\@slowromancap\romannumeral #1@}
\makeatother

\theoremstyle{definition}

\newtheorem{fact}{fact}

\newtheorem{thm}[fact]{Theorem}
\newtheorem{lemma}[fact]{Lemma}
\newtheorem{prop}[fact]{Proposition}
\newtheorem{corollary}[fact]{Corollary}
\newtheorem{defini}[fact]{Definition}

\title{Algorithmic Randomness for Infinite Time Register Machines}
\author{Merlin Carl}
\date{}

\begin{document}

\maketitle

\begin{abstract}
A concept of randomness for infinite time register machines ($ITRM$s), resembling Martin-L\"of-randomness, is defined and studied. In particular, we show that for this notion of randomness,
 computability from mutually random reals implies computability and that an analogue of van Lambalgen's theorem holds.
\end{abstract}

\section{Introduction}

Martin-L\"of-randomness ($ML$-randomness, see e.g. \cite{DoHi}) provides an intuitive and conceptually stable clarification of the informal notion of a random sequence over a finite alphabet.
Since its introduction, several strengthenings and variants of $ML$-randomness have been considered; a recent example is the work of Hjorth and Nies on $\Pi_{1}^{1}$-randomness, which led to
interesting connections with descriptive set theory (\cite{HN}).\\
We are interested in obtaining a similar notion based on machine models of transfinite computations.  In this paper, we will exemplarily consider infinite time register machines.
Infinite Time Register Machines ($ITRM$s), introduced in \cite{ITRM} and further studied in \cite{ITRM2}, work similar to the classical unlimited register machines described in \cite{Cu}. 
In particular, they use finitely many registers each of which can store a single natural number. The difference is that $ITRM$s use transfinite ordinal running time: The state of an $ITRM$
at a successor ordinal is obtained as for $URM$s. At limit times, the program line is the limit inferior of the earlier program lines and there is a similar limit rule for the register contents. 
If the limes inferior of the earlier register contents is infinite, the register is reset to $0$.

The leading idea of $ML$-randomness is that a sequence of $0$ and $1$ is random iff it has no special properties, where a special property should be a small (e.g. measure $0$) set of reals
that is in some way accessible to a Turing machine.
Classical Turing machines, due to the finiteness of their running time, have the handicap that the only decidable null set of reals is the empty set. In the definition of $ML$-randomness,
this difficulty is overcome by merely demanding the set $X$ to be effectively approximated by a recursively enumerable sequence of sets of intervals with controlled convergence behaviour.
This limits the sense in which this randomness notion is effective, which was the motivation for Schnorr's criticism of $ML$-randomness (see e.g. \cite{Sc}).
For models of transfinite computations, this trick is unnecessary and this criticism can be entirely avoided:
The decidable sets of reals form a rich class (particularly including all sets approximated by $ML$-tests), while plausibility is retained as reals in a set decidable by such a machine
can still be reasonably said to have a special property.
Hence, we define:

\begin{defini}{\label{randomdef}} Recall that a set $X\subseteq\mathfrak{P}(\omega)$ is meager iff it is a countable union of nowhere dense sets.
$x\subseteq\omega$ is $ITRM$-random iff there is no $ITRM$-decidable, meager set $X\subseteq\mathfrak{P}(\omega)$ such that $x\in X$.
\end{defini}

This obviously deviates from the definition of $ML$-randomness since we use meager sets rather than null sets as our underlying notion of `small'. The reason is simply
that this variant turned out to be much more convenient to handle for technical reasons. We are pursuing the analogous notion for null sets in ongoing work \cite{CaSc2}.

We will now summarize some key notions and results on $ITRM$s that will be used in the paper.

\begin{defini}
For $P$ a program, $x,y\in\mathfrak{P}(\omega)$, 
$P^{x}\downarrow=y$ means that the program $P$, when run with oracle $x$, holds on every input $i\in\omega$ and outputs $1$ iff $i\in y$ and $0$, otherwise.
$x\subseteq\omega$ is $ITRM$-computable in the oracle $y\subseteq\omega$ iff there is an $ITRM$-program $P$ such that $P^{y}\downarrow=x$, in which case we occasionally write $x\leq_{ITRM}y$. If $y$ can be taken to be
$\emptyset$, $x$ is $ITRM$-computable. We denote the set of $ITRM$-computable reals by $COMP$.
\end{defini}

\textbf{Remark}: We occasionally drop the $ITRM$-prefix as notions like `computable' always refer to $ITRM$s in this paper.

\begin{thm}{\label{relITRM}}
 Let $x,y\subseteq\omega$. Then $x$ is $ITRM$-computable in the oracle $y$ iff $x\in L_{\omega_{\omega}^{CK,y}}[y]$, where $\omega_{i}^{CK,x}$ denotes the $i$th $x$-admissible ordinal.
\end{thm}
\begin{proof}
This is a straightforward relativization of the main result of \cite{ITRM}.
\end{proof}

\begin{thm}{\label{hp}}
Let $\mathbb{P}_{n}$ denote the set of $ITRM$-programs using at most $n$ registers, and let $(P_{i,n}|i\in\omega)$ enumerate $\mathbb{P}_{n}$ in some natural way. 
Then the bounded halting problem $H_{n}^{x}:=\{i\in\omega|P_{i,n}^{x}\downarrow\}$ is computable uniformly in the oracle $x$ by an $ITRM$-program.\\
Furthermore, if $P\in\mathbb{P}_{n}$ and $P^{x}\downarrow$, then $P^{x}$ halts in less than $\omega_{n+1}^{CK,x}$ many steps. 
Consequently, if $P$ is a halting $ITRM$-program, then $P^{x}$ stops in less than $\omega_{\omega}^{CK,x}$ many steps.
\end{thm}
\begin{proof}
 The corresponding results from \cite{ITRM} easily relativize. 
\end{proof}

We will freely use the following standard proposition:

\begin{prop}
 Let $X\subseteq[0,1]\times[0,1]$ and $\tilde{X}:=\{x\oplus y\mid (x,y)\in X\}$. Then $X$ is meager/comeager/non-meager iff $\tilde{X}$ is meager/comeager/non-meager.
\end{prop}
\begin{proof}\end{proof}

Most of our notation is standard. 
$L_{\alpha}[x]$ denotes the $\alpha$th level of G\"odel's constructible hierarchy relativized to $x$. For $a,b\subseteq\omega$, $a\oplus b$ denotes $\{p(i,j)\mid i\in a\wedge j\in b$, where
$p:\omega\times\omega\rightarrow\omega$ is Cantor's pairing function.

\section{Computability from random oracles}

In this section, we consider the question which reals can be computed by an $ITRM$ with an $ITRM$-random oracle. We start by recalling
 the following theorem from \cite{CaSc}. The intuition behind it is that, given a certain non-$ITRM$-real $x$, one has no chance of computing it from
some randomly chosen real $y$.

\begin{thm}{\label{ManyOracles}}
Let $x$ be a real, $Y$ be a set of reals such that $x$ is $ITRM$-computable from every $y\in Y$.\\
Then, if $X$ has positive Lebesgue measure or is Borel and non-meager, $x$ is $ITRM$-computable.
\end{thm}

\begin{corollary}{\label{Admpreserv}}
 Let $x$ be $ITRM$-random. Then, for all $i\in\omega$, $\omega_{i}^{CK,x}=\omega_{i}^{CK}$.
\end{corollary}
\begin{proof}
Lemma $46$ of \cite{CaSc} shows that $\omega_{i}^{CK,x}=\omega_{i}^{CK}$ for all $i\in\omega$ whenever $x$ is Cohen-generic over $L_{\omega_{\omega}^{CK}}$ (see e.g. \cite{CaSc} or \cite{Ma})
 and that the set of Cohen-generics over $L_{\omega_{\omega}^{CK}}$
is comeager. Hence $\{x|\omega_{i}^{CK,x}>\omega_{i}^{CK}\}$ is meager.
For each program $P$, the set of reals $x$ such that $P^{x}$ computes a code for the $i$th $x$-admissible which is greater than $\omega_{i}^{CK}$ is decidable using the techniques 
developed in \cite{Ca} and \cite{Ca2}. (The idea is to uniformly in the oracle $x$ compute a real $c$ coding $L_{\omega_{i+1}^{CK,x}}[x]$ in which the natural numbers $m$ and $n$ coding
$\omega_{i}^{CK}$ and $\omega_{i}^{CK,x}$ can be identified in the oracle $x$, then check - using a halting problem solver
for $P$, see Theorem \ref{hp} - whether $P^{x}$ computes a well-ordering of the same order type as the element of $L_{\omega_{i+1}^{CK,x}}[x]$ coded by $k$ and 
finally whether the element coded by $m$ is an element of that coded by $n$.)
Hence, if $x$ is $ITRM$-random, then there can be no $ITRM$-program $P$ computing such a code in the oracle $x$. But a code for $\omega_{i}^{CK,x}$ is
$ITRM$-computable in the oracle $x$ for every real $x$ and every $i\in\omega$. Hence, we must have $\omega_{i}^{CK,x}=\omega_{i}^{CK}$ for every $i\in\omega$, 
as desired.
\end{proof}

\begin{lemma}{\label{compandgen}}
Let $a\subseteq\omega$ and suppose that $z$ is Cohen-generic over $L_{\omega_{\omega}^{CK,a}+1}[a]$. Then $a\leq_{ITRM}z$ iff $a$ is $ITRM$-computable.
Moreover, the set $C_{a}:=\{z\subseteq\omega\mid z\text{ is Cohen-generic over }L_{\omega_{\omega}^{CK,a}+1}[a]\}$ is comeager. Consequently, $S_{a}:=\{z\subseteq\omega\mid a\leq_{ITRM}z\}$ is meager whenever $a$ is not $ITRM$-computable.
\end{lemma}
\begin{proof}
Assume that $z$ is Cohen-generic over $L_{\omega_{\omega}^{CK,a}+1}[a]$ and $a\leq_{ITRM}z$. By the forcing theorem for provident sets (see e.g. Lemma $32$ of \cite{CaSc}), there is an $ITRM$-program $P$ and
a forcing condition $p$ such that $p\Vdash P^{\dot{G}}\downarrow=$\v{a}, where $\dot{G}$ is the canonical name for the generic filter and \v{a} is the canonical name of $a$.
 Now, let $y$ and $z$ be mutually Cohen-generic over $L_{\omega_{\omega}^{CK,a}+1}[a]$ both extending $p$.
Again by the forcing theorem and by absoluteness of computations, we must have $P^{x}\downarrow=a=P^{y}\downarrow$, so $a\in L_{\omega_{\omega}^{CK,x}}[x]\cap L_{\omega_{\omega}^{CK,y}}[y]$.
By Corollary \ref{Admpreserv}, $\omega_{\omega}^{CK,x}=\omega_{\omega}^{CK,y}=\omega_{\omega}^{CK}$. By Lemma $28$ of \cite{CaSc}, we have $L_{\alpha}[x]\cap L_{\alpha}[y]=L_{\alpha}$ whenever
$x$ and $y$ are mutually Cohen-generic over $L_{\alpha}$ and $\alpha$ is provident (see \cite{Ma}). Consequently, we have\\ $a\in L_{\omega_{\omega}^{CK,x}}[x]\cap L_{\omega_{\omega}^{CK,y}}[y]=L_{\omega_{\omega}^{CK}}[x]\cap L_{\omega_{\omega}^{CK}}[y]
=L_{\omega_{\omega}^{CK}}$, so $a$ is $ITRM$-computable.\\
The comeagerness of $C_{a}$ is standard (see e.g. Lemma $29$ of \cite{CaSc}). To see that $S_{a}$ is meager for non-$ITRM$-computable $a$, observe that the Cohen-generic reals over $L_{\omega_{\omega}^{CK,a}+1}[a]$
form a comeager set of reals to non of which $a$ is reducible.
\end{proof}

\begin{defini}
Let $x,y\subseteq\omega$. Then $x$ is $ITRM$-random relative to $y$ iff there is no meager set $X$ such that $x\in X$ and $X$ is $ITRM$-decidable in the oracle $y$.
If $x$ is $ITRM$-random relative to $y$ and $y$ is $ITRM$-random relative to $x$, we say that $x$ and $y$ are mutually $ITRM$-random.
\end{defini}

Intuitively, we should expect that mutually random reals have no non-trivial information in common. This is expressed by the following theorem:

\begin{thm}{\label{mutuallyrandom}}
 If $z$ is $ITRM$-computable from two mutually $ITRM$-random reals $x$ and $y$, then $z$ is $ITRM$-computable.
\end{thm}
\begin{proof}
Assume otherwise, and suppose that $z$, $x$ and $y$ constitute a counterexample. By assumption, $z$ is computable from $x$. Also, let $P$ be a program such that $P^{a}(i)\downarrow$ for every $a\subseteq\omega$, $i\in\omega$ and such that $P$
computes $z$ in the oracle $y$. In the oracle $z$, the set $A_{z}:=\{a|\forall{i\in\omega}P^{a}(i)\downarrow=z(i)\}$ is decidable by simply computing $P^{a}(i)$ for all $i\in\omega$ and comparing the result to the $i$th bit of $z$.
Clearly, we have $A_{z}\subseteq\{a\mid z\leq_{ITRM}a\}$. Hence, by our Theorem \ref{compandgen} above, $A_{z}$ is meager as $z$ is not $ITRM$-computable by assumption. 
Since $A_{z}$ is decidable in the oracle $z$ and $z$ is computable from $x$, $A_{z}$ is also decidable in the oracle $x$. Now, $x$ and $y$ are 
mutually $ITRM$-random, so that $y\notin A_{z}$. But $P$ computes $z$ in the oracle $y$, so $y\in A_{z}$ by definition, a contradiction.
\end{proof}

While, naturally, there are non-computable reals that are reducible to a random real $x$ (such as $x$ itself), intuitively, it should not be possible to compute a non-arbitrary real from a random real.
We approximate this intuition by taking `non-arbitrary' to mean `$ITRM$-recognizable' (see \cite{ITRM2} or \cite{Ca} for more information on $ITRM$-recognizability).
 It turns out that, in accordance with this intuition, recognizables that are $ITRM$-computable from $ITRM$-random reals are already $ITRM$-computable.

\begin{defini}
$x\subseteq\omega$ is $ITRM$-recognizable iff there is an $ITRM$-program $P$ such that $P^{y}\downarrow=1$ iff $y=x$ and $P^{y}\downarrow=0$, otherwise.
\end{defini}

\begin{thm}{\label{weakrandomreducibility}}
 Let $x\in RECOG$ and let $y$ be $ITRM$-random such that $x\leq_{ITRM}y$. Then $x$ is $ITRM$-computable.
\end{thm}
\begin{proof}
Let $x\in RECOG-COMP$ be computable from $y$, say by program $P$ and let $Q$ be a program that recognizes $x$. 
The set $S:=\{z\mid P^{z}\downarrow=x\}$ is meager as in the proof of Theorem \ref{mutuallyrandom}. But $S$ is decidable:
Given a real $z$, use a halting-problem solver for $P$ (which exists uniformly in the oracle by Theorem \ref{hp}) to test whether $P^{z}(i)\downarrow$ for all $i\in\omega$; if not,
then $z\notin S$. Otherwise, use $Q$ to check whether the real computed by $P^{z}$ is equal to $x$. If not, then $z\notin S$, otherwise $z\in S$. As $P^{y}$ computes $x$,
it follows that $y\in S$, so that $y$ is an element of an $ITRM$-decidable meager set. Hence $y$ is not $ITRM$-random, a contradiction.
\end{proof}

\textbf{Remark}: Let $(P_{i}|i\in\omega)$ be a natural enumeration of the $ITRM$-programs.
 Together with the fact that the halting number $h=\{i\in\omega\mid P_{i}\downarrow\}$ for $ITRM$s is recognizable (see \cite{Ca2}), this implies in particular that the halting problem for $ITRM$s is not 
$ITRM$-reducible to an $ITRM$-random real.\\

It follows from Theorem \ref{hp} that the computational strength of $ITRM$s increases strictly with the number of registers. Let us say that $x$ is $ITRM$-$n$-random for $n\in\omega$
iff there is no meager set $X\ni x$ which is decidable by an $ITRM$ using at most $n$ registers. One should expect that stronger information extraction methods lead to a stronger notion or randomness, which is indeed the case:
The following theorem, together with Theorem \ref{mutuallyrandom}, shows that this notion becomes strictly
stronger with the number of registers. It also implies that for each $n$, there is $m>n$ such that a non-$ITRM_m$-computable real $x$ is
$ITRM_m$-computable from two mutually $ITRM_n$-random reals $y$ and $z$ and consequently, that there is no universal test for $ITRM$-randomness, i.e. the set of $ITRM$-random reals is not $ITRM$-decidable.

\begin{thm}{\label{RegNum}}
For each $n\in\omega$, there are $x,y,z\subseteq\omega$ such that the following holds:
$y$ and $z$ are mutually random for $ITRM$s with $n$ registers, $x$ is not $ITRM$-computable and
$x$ is $ITRM$-computable from $y$ and $z$. In fact, for every constructible real $x$ and every $n\in\omega$, there are reals
$y$ and $z$ and $m\in\omega$ such that $y$ and $z$ are mutually $ITRM_n$-random and $x$ is $ITRM_m$-computable
both from $y$ and $z$.
\end{thm}
\begin{proof}
(1) Let $x:=cc(L_{\omega_{\omega+1}^{CK}})$ (the $<_{L}$-minimal real coding $L_{\omega_{\omega+1}^{CK}}$) and 
pick $y$ and $z$ mutually Cohen-generic over $L_{\omega_{n+1}^{CK}}$ such that $\omega_{n+2}^{CK,y}=\omega_{n+2}^{CK,z}=\omega_{\omega+1}^{CK}$.
Then $x$ is certainly $ITRM$-computable both from $y$ and $z$, and $y$ and $z$ are mutually $ITRM_n$-random. Also, $x\notin L_{\omega_{\omega}^{CK}}$, so 
$x$ is not $ITRM$-computable.\\
(2) Let $\alpha$ be the smallest admissible with $x\in L_{\alpha}$. Clearly, $x$ is $ITRM_m$-computable from $cc(L_{\alpha})$ for some $m_{1}\in\omega$.
Now, pick $y$ and $z$ mutually Cohen-generic over $L_{\omega_{n+1}^{CK}}$ such that $\omega_{n+2}^{CK,y}=\omega_{n+2}^{CK,z}=\omega_{\omega+1}^{CK}$.
As above, $y$ and $z$ are mutually $ITRM_n$-random and $x$ is $ITRM$-computable both from $y$ and $z$, hence $ITRM_{m_{2}}$-computable from $y$
and $ITRM_{m_{3}}$-computable from $z$ for some $m_{2},m_{3}\in\omega$. Taking $m=max\{m_{1},m_{2},m_{3}\}$, we obtain the desired result.
\end{proof}

\section{An analogue to van Lambalgen's theorem}

A crucial result of classical algorithmic randomness is van Lambalgen's theorem, which
 states that for reals $a$ and $b$, $a\oplus b$ is $ML$-random iff $a$ is $ML$-random and $b$ is $ML$-random relative to $a$.
In this section, we demonstrate an analogous result for $ITRM$-randomness.

\begin{defini}
 $X\subseteq\mathfrak{P}(\omega)$ is called $ITRM$-decidable iff there is an $ITRM$-program $P$ such that $P^{x}\downarrow=1$ iff $x\in X$ and $P^{x}\downarrow=0$, otherwise. In this case we say that $P$ decides $X$.
$P$ is called deciding iff there is some $X$ such that $P$ decides $X$. We say that $X$ is decided by $P$ in the oracle $y$ iff $X=\{x\mid P^{x\oplus y}\downarrow=1\}$ and $\mathfrak{P}(\omega)-X=\{x\mid P^{x\oplus y}\downarrow=0\}$.
The other notions relativize in the obvious way.
\end{defini}

%

\begin{lemma}
 Let $Q$ be a deciding $ITRM$-program using $n$ registers and $a\subseteq\omega$. Then $\{y|Q^{y\oplus a}\downarrow=1\}$ is meager iff
$Q^{x\oplus a}\downarrow=0$ for all $x\in L_{\omega_{n+1}^{CK,a}+3}[a]$ that are Cohen-generic over $L_{\omega_{n+1}^{CK,a}+1}[a]$.
\end{lemma}
\begin{proof}
By absoluteness of computations and the bound on $ITRM$-halting times (see Theorem \ref{hp}), 
$Q^{x\oplus a}\downarrow=0$ implies that $Q^{x\oplus a}\downarrow=0$ also holds in $L_{\omega_{n+1}^{CK,a}}[a]$. As this is expressable by a $\Sigma_1$-formula, it must be forced by some condition $p$ by the 
forcing theorem over $KP$ (see e.g. Theorem $10.10$ of \cite{Ma}).

Hence every $y$ extending $p$ will satisfy $Q^{y\oplus a}\downarrow=0$. The set $C$ of reals Cohen-generic over $L_{\omega_{n+1}^{CK,a}+1}[a]$ 
is comeager. Hence, if $Q^{x\oplus a}\downarrow=0$ for some $x\in C$, then $Q^{x\oplus a}\downarrow=0$ for a non-meager (in fact comeager in some interval) set $C^{\prime}$.
Now, for each condition $p$, $L_{\omega_{n+1}^{CK,a}+3}[a]$ will contain a generic filter over $L_{\omega_{n+1}^{CK,a}+1}[a]$ extending $p$ (as $L_{\omega_{n+1}^{CK,a}+1}[a]$ is countable
in $L_{\omega_{n+1}^{CK,a}+3}[a]$). 
Hence, if $Q^{x\oplus a}\downarrow=0$ for all $x\in C\cap L_{\omega_{n+1}^{CK,a}+3}[a]$, then this holds for all elements of $C$ and the complement $\{y|Q^{y\oplus a}\downarrow=1\}$ is therefore meager.

If, on the other hand, $Q^{x\oplus a}\downarrow=1$ for some such $x$, then this already holds for all $x$ in some non-meager (in fact comeager in some interval) set $C^{\prime}$ by the same reasoning. 
\end{proof}

\begin{corollary}{\label{decidemeasure}}
 For a deciding $ITRM$-program $Q$ using $n$ registers, there exists an $ITRM$-program $P$ such that, for all $x,y\in\mathfrak{P}(\omega)$, $P^{x}\downarrow=1$ iff $\{y|Q^{x\oplus y}\downarrow=1\}$ is of non-meager.
\end{corollary}
\begin{proof}
 From $x$, compute, using sufficiently many extra registers, a real code for $L_{\omega_{n+1}^{CK,x}+4}$ in the oracle $x$. This can be done uniformly in $x$. Then, using the techniques developed in section $6$ of \cite{ITRM}, 
identify and check all generics in that structure,
 according to the last lemma. 
\end{proof}

 Our proof of the
$ITRM$-analogue for van Lambalgen's theorem now follows a general strategy inspired by that used in \cite{DoHi}, Theorem $6.9.1$ and $6.9.2$: 

\begin{thm}{\label{Dir1}}
 Assume that $a$ and $b$ are reals such that $a\oplus b$ is not $ITRM$-random. Then $a$ is not $ITRM$-random or $b$ is not $ITRM$-random relative to $a$.
\end{thm}
\begin{proof}
As $a\oplus b$ is not $ITRM$-random, let $X$ be an $ITRM$-decidable meager set of reals such that $a\oplus b\in X$. Suppose that $P$ is a program deciding $X$.\\
Let $Y:=\{x|\{y\mid x\oplus y\in X\}\text{ non-meager}\}$. By Corollary \ref{decidemeasure}, $Y$ is $ITRM$-decidable.\\
We claim that $Y$ is meager. First, $Y$ is provably $\Delta_{1}^{1}$ and hence has the Baire property (see e.g. Exercise $14.5$ of \cite{Ka}). Hence, by
the Kuratowski-Ulam-theorem (see e.g. \cite{Ke}, Theorem $8.41$), $Y$ is meager. Consequently, if $a\in Y$, then $a$ is not $ITRM$-random.\\

Now suppose that $a\notin Y$. This means that $\{y\mid a\oplus y\in X\}$ is meager. But $S:=\{y\mid a\oplus y\in X\}$ is easily seen to be $ITRM$-decidable in the oracle $a$ and
$b\in S$. Hence $b$ is not $ITRM$-random relative to $a$.
\end{proof}

\begin{thm}{\label{Dir2}}
 Assume that $a$ and $b$ are reals such that $a\oplus b$ is $ITRM$-random. Then $a$ is $ITRM$-random and $b$ is $ITRM$-random relative to $a$.
\end{thm}
\begin{proof}
 Assume first that $a$ is not $ITRM$-random, and let $X$ be an $ITRM$-decidable meager set with $a\in X$.
Then $X\oplus[0,1]$ is also meager and $X\oplus [0,1]$ is $ITRM$-decidable. As $a\in X$, we have $a\oplus b\in X\oplus[0,1]$, so
$a\oplus b$ is not $ITRM$-random, a contradiction.\\
Now suppose that $b$ is not $ITRM$-random relative to $a$, and let $X$ be a meager set of reals such that $b\in X$ and $X$ is $ITRM$-decidable in the oracle $a$.
Let $Q$ be an $ITRM$-program such that $Q^{a}$ decides $X$. Our goal is to define a deciding program $\tilde{Q}$ such that $\tilde{Q}^{a}$ still decides $X$, but
also $\{x|\tilde{Q}^{x}\downarrow=1\}$ is meager. This suffices, as then $\tilde{Q}^{a\oplus b}\downarrow=1$ and $\{x|\tilde{Q}^{x}\downarrow=1\}$ is $ITRM$-decidable.
$\tilde{Q}$ operates as follows: Given $x=y\oplus z$, check whether $\{w\mid Q^{y\oplus w}\}$ is meager, using Corollary \ref{decidemeasure}. If that is the case,
carry out the computation of $Q^{x}$ and return the result. Otherwise, return $0$. This guarantees (since $X$ is meager) 
that $\{y\mid \tilde{Q}^{x\oplus y}\downarrow=1\}$ is meager and furthermore that $\tilde{Q}^{a\oplus x}\downarrow=1$ iff $Q^{a\oplus x}\downarrow=1$ iff
$x\in X$ for all reals $x$, so that $\{x|\tilde{Q}^{a\oplus x}\downarrow=1\}$ is just $X$, as desired.
\end{proof}

Combining Theorem \ref{Dir1} and \ref{Dir2} gives us the desired conclusion:

\begin{thm}{\label{vLamb}}
 Given reals $x$ and $y$, $x\oplus y$ is $ITRM$-random iff $x$ is $ITRM$-random and $y$ is $ITRM$-random relative to $x$.
In particular, if $x$ and $y$ are $ITRM$-random, then $x$ is $ITRM$-random relative to $y$ iff $y$ is $ITRM$-random relative to $x$.
\end{thm}

We note that a classical Corollary to van Lambalgen's theorem continues to hold in our setting:

\begin{corollary}
 Let $x,y$ be $ITRM$-random. Then $x$ is $ITRM$-random relative to $y$ iff $y$ is $ITRM$-random relative to $x$.
\end{corollary}
\begin{proof}
 Assume that $y$ is $ITRM$-random relative to $x$. By assumption, $x$ is $ITRM$-random. By Theorem \ref{vLamb}, $x\oplus y$ is $ITRM$-random. Trivially, $y\oplus x$ is also $ITRM$-random.
Again by Theorem \ref{vLamb}, $x$ is $ITRM$-random relative to $y$. By symmetry, the corollary holds.
\end{proof}

\section{Some consequences for the structure of $ITRM$-degrees}

In the new setting, we can also draw some standard consequences of van Lambalgen's theorem. 

\begin{defini}
If $x\leq_{ITRM}y$ but not $y\leq_{ITRM} x$, we write $x<_{ITRM}y$. If $x\leq_{ITRM}$ and $y\leq_{ITRM}x$, then we write $x\equiv_{ITRM}y$.
\end{defini}

Clearly, $\equiv_{ITRM}$ is an equivalence relation.
We may hence form, for each real $x$, the $\equiv_{ITRM}$-equivalence class $[x]_{ITRM}$ of $x$, called the $ITRM$-degree of $x$. It is easy to see that $\leq_{ITRM}$ respects $\equiv_{ITRM}$,
so that $[x]_{ITRM}\leq_{ITRM}[y]_{ITRM}$ etc. are well-defined and have the obvious meaning.

\begin{corollary}{\label{ITRMsplitting}}
 If $a$ is $ITRM$-random, $a=a_{0}\oplus a_{1}$, then $a_{0}\nleq_{ITRM} a_{1}$ and $a_{1}\nleq_{ITRM}a_{0}$.
\end{corollary}
\begin{proof}
 By Theorem \ref{vLamb}, $a_{0}$ and $a_{1}$ are mutually $ITRM$-random. If $a_{0}$ was $ITRM$-computable from $a_{1}$, then $\{a_{0}\}$ would be decidable
in the oracle $a_{1}$, meager and contain the $ITRM$-random real $a_{0}$, a contradiction. By symmetry, the claim follows.
\end{proof}

\begin{lemma}{\label{randomnessandhalting}}
Let $h$ be a real coding the halting problem for $ITRM$s as in the remark following Theorem \ref{weakrandomreducibility}. Then there is an $ITRM$-random real $x\leq_{ITRM} h$.
\end{lemma}
\begin{proof}
 Given $h$, we can compute a code for $L_{\omega_{\omega}^{CK}+2}$, which contains a real $x$ which is Cohen-generic over $L_{\omega_{\omega}^{CK}+1}$.
Hence, $x$ itself is $ITRM$-computable from $h$. Assume that $x$ is not $ITRM$-random, so there exists a decidable meager set $X\ni x$.
Let $P$ be a program deciding $X$. Then $P^{x}\downarrow=1$. By the forcing theorem for Cohen-forcing, this must be forced by some condition $p\subseteq x$.
The set $Y$ of $y\supseteq p$ which are Cohen-generic over $L_{\omega_{\omega}^{CK}+1}$ is non-meager (see above) and $p\subseteq y$ implies $p\Vdash P^{y}\downarrow=1$.
As $P$ decides $X$, we must have $Y\subseteq X$, a contradiction to the assumption that $X$ is meager.
\end{proof}

As a corollary, we obtain an analogue solution to the Kleene-Post-theorem on Turing degrees between $0$ and $0^{\prime}$ (see e.g. Theorem VI.1.2 of \cite{So}) for $ITRM$s. 

\begin{corollary}{\label{ITRMPost}}
With $h$ as in Lemma \ref{randomnessandhalting}, there is a real $y$ such that $[0]_{ITRM}<_{ITRM}[y]_{ITRM}<_{ITRM}h$.
\end{corollary}
\begin{proof}
Pick $x$ as in Lemma \ref{randomnessandhalting}, and let $x=x_{0}\oplus x_{1}$. Obviously, we have $x_{0},x_{1}\leq_{ITRM}x$. By Corollary \ref{ITRMsplitting}, $[x_{0}]_{ITRM}\neq[x_{1}]_{ITRM}$,
hence $[x_{0}]_{ITRM}\neq[h]_{ITRM}$ or $[x_{1}]_{ITRM}\neq[h]_{ITRM}$. Assume without loss of generality that the former holds. Then $[x_{0}]_{ITRM}<_{ITRM}[h]_{ITRM}$.
As $x_{0}$ is $ITRM$-random by Theorem \ref{vLamb}, $x_{0}$ is not $ITRM$-computable, so $[0]_{ITRM}<_{ITRM}[x_{0}]_{ITRM}$. Thus $[x_{0}]_{ITRM}$ is as desired.
\end{proof}

\section{Conclusion and further work}
The most pressing issue is certainly to strengthen the parallelism between $ITRM$-randomness and $ML$-randomness by studying the corresponding notion for sets of Lebesgue measure $0$ rather than meager sets.\\
Still, $ITRM$-randomness in its current form shows an interesting behaviour, partly analogous to $ML$-randomness, though by quite different arguments. Similar approaches are likely to work for other machine models of generalized computations,
in particular $ITTM$s (\cite{HaLe}) (which were shown in \cite{CaSc} to obey the analogue of the non-meager part of Theorem \ref{ManyOracles}) and ordinal Turing Machines (\cite{Ko}) (for which the analogues of both parts of Theorem 
\ref{ManyOracles} turned out to be independent from $ZFC$) which we study in ongoing work (\cite{CaSc2}). This further points towards a more general background theory of computation that allow unified arguments 
for all these various models as well as classical computability.
Furthermore, we want to see whether the remarkable conceptual stability of $ML$-randomness 
(for example the equivalence with Chaitin randomness or unpredictabiliy in the sense of r.e. Martingales, see e.g. sections $6.1$ and $6.3$ of \cite{DoHi})
carries over to the new context.

\section{Acknowledgements}

I am indebted to Philipp Schlicht for several helpful discussions of the results and proofs and suggesting various crucial references.

\end{document}